\documentclass[11pt, letterpaper]{article}

\usepackage{fullpage}
\usepackage{colordvi,epsfig}
\usepackage{amsfonts,amssymb,amsmath,amsthm}
\usepackage{algorithm}
\usepackage[noend]{algpseudocode}
\usepackage{caption}
\usepackage{subcaption}
\graphicspath{{./figures/}}  
\usepackage{hyperref}

\usepackage[maxbibnames=99, maxcitenames=10,doi=false,isbn=false,url=false,backend=bibtex]{biblatex}	

\bibliography{SDA.bib}

\newcommand{\ignore}[1]{}
\newcommand{\R}{\mathbb{R}}

\newcommand{\eqdef}{\overset{\text{def}}{=}} 

\newcommand{\E}[1]{\mathbf{E}\left[#1\right] } 
\newcommand{\norm}[1]{\lVert#1\rVert}

\newcommand{\Tr}[1]{\mathbf{Tr}\left( #1\right)}
\providecommand{\Null}[1]{\mathbf{Null}\left( #1\right)}
\providecommand{\Rank}[1]{\mathbf{Rank}\left( #1\right)}
\providecommand{\Range}[1]{\mathbf{Range}\left( #1\right)}

\usepackage[colorinlistoftodos,bordercolor=orange,backgroundcolor=orange!20,linecolor=orange,textsize=scriptsize]{todonotes}

\newtheorem{proposition}{Proposition} [section]

\newtheorem{theorem}{Theorem}  [section]
\newtheorem{lemma}{Lemma} [section]
\newtheorem{corollary}{Corollary} [section]
\newtheorem{remark}{Remark} [section]

\title{Stochastic Dual Ascent for Solving Linear Systems} 

\author{Robert M. Gower and Peter Richt\'{a}rik\thanks{This author would like to acknowledge support from the EPSRC Grant EP/K02325X/1, ``Accelerated Coordinate Descent Methods for Big Data Optimization'' and EPSRC Fellowship Grant EP/N005538/1, ``Randomized Algorithms for Extreme Convex Optimization''. }\\\\
\it School of Mathematics\\ \it University of Edinburgh\\\it United Kingdom}

\date{December 21, 2015}

\begin{document}
\maketitle

\begin{abstract}{\footnotesize
We develop a new randomized iterative algorithm---{\em stochastic dual ascent (SDA)}---for finding the projection of a given vector onto the solution space of a  linear system. The method is dual in nature: with the dual being a non-strongly concave quadratic maximization problem without constraints. In each iteration of SDA, a dual variable is updated by a carefully chosen point in a subspace  spanned by the columns of a random matrix drawn independently from a fixed distribution. The distribution plays the role of a parameter of the method.  Our complexity results hold for a wide family of distributions of random matrices, which opens the possibility to fine-tune the stochasticity of the method to particular applications.  We prove that  primal iterates associated with the dual process converge to the projection exponentially fast in expectation, and give a formula and an insightful lower bound for the convergence rate. We also prove that the same rate applies to dual function values, primal function values and the duality gap. Unlike traditional iterative methods, SDA converges under no additional assumptions on the system (e.g., rank, diagonal dominance) beyond consistency. In fact, our lower bound  improves as the rank of the system matrix drops.   Many existing randomized methods for linear systems arise as special cases of SDA, including randomized Kaczmarz, randomized Newton, randomized coordinate descent, Gaussian descent, and their variants. In special cases where our method specializes to a known algorithm, we either recover the best known rates, or improve upon them. Finally, we show that the framework can be applied to the distributed average consensus problem to obtain  an array of new algorithms. The randomized gossip algorithm arises as a special case. 
}
\end{abstract}


\section{Introduction}

Probabilistic ideas and tools have recently begun to permeate into several fields where they had traditionally not played a major role,  including fields such as  numerical linear algebra and optimization. One of the key ways in which these ideas influence these fields is via the development and analysis of  {\em randomized algorithms} for solving standard and new problems of these fields. Such methods are typically easier to analyze, and often lead to faster and/or more scalable and versatile methods in practice.

\subsection{The problem}

In this paper we  consider a key problem in linear algebra, that of finding a solution of  a system of linear equations
\begin{equation}\label{eq:Axbx}Ax =b,\end{equation}
where $A \in \R^{m \times n}$  and $b \in \R^m$. We shall assume throughout that the system is {\em consistent}, that is,  that  there exists $x^*$ for which $Ax^*=b$.  While we assume the existence of a solution, we do not assume uniqueness. In situations with multiple solutions, one is often interested in finding a solution with specific properties. For instance, in compressed sensing and sparse optimization, one is interested in finding the least $\ell_1$-norm, or the least $\ell_0$-norm (sparsest) solution. 

In this work we shall focus on the canonical problem of finding the solution  of \eqref{eq:Axbx} closest, with respect to a Euclidean distance,  to a given vector $c\in \R^n$:
\begin{eqnarray}  \text{minimize} && P(x)\eqdef \tfrac{1}{2}\|x-c\|_B^2 \notag\\
\text{subject to} && Ax=b \label{eq:P}\\
&& x\in \R^n.\notag
\end{eqnarray}
where $B$ is an $n\times n$ symmetric positive definite matrix and 
$\|x\|_B \eqdef \sqrt{x^\top B x}$. By $x^*$ we denote the (necessarily) unique solution of \eqref{eq:P}. Of key importance in this paper is the {\em dual problem}\footnote{Technically, this  is both the Lagrangian and Fenchel dual of \eqref{eq:P}.} to \eqref{eq:P}, namely
\begin{eqnarray}\label{eq:D} \text{maximize} && D(y)\eqdef (b-Ac)^\top y - \tfrac{1}{2}\|A^\top y\|_{B^{-1}}^2\\
\text{subject to}&& y \in \R^m. \notag
\end{eqnarray}

Due to the consistency assumption, strong duality holds and we have $P(x^*) = D(y^*)$, where $y^*$ is any dual optimal solution.

\subsection{A new family of stochastic optimization algorithms}

We propose to solve \eqref{eq:P} via a new method operating in the dual \eqref{eq:D}, which we call {\em stochastic  dual ascent} (SDA). The iterates of SDA are of the form
\begin{equation}\label{eq:methoddual0} y^{k+1} = y^k + S \lambda^k,\end{equation}
where $S$ is a random matrix with $m$ rows  drawn in each iteration independently from a pre-specified distribution ${\cal D}$, which should be seen as a parameter of the method. In fact,  by varying ${\cal D}$, SDA should be seen as a family of algorithms indexed by $\cal D$, the choice of which leads to specific algorithms in this family.   By performing steps of the form \eqref{eq:methoddual0}, we are moving in the range space of the random matrix $S$. A key feature of SDA enabling us to prove strong  convergence results despite the fact that the dual objective is in general not strongly concave is the way in which  the ``stepsize'' parameter $\lambda^k$ is chosen: we chose $\lambda^k$ to be the {\em least-norm} vector for which $D(y^k + S\lambda)$ is maximized in $\lambda$. Plugging this $\lambda^k$ into~\eqref{eq:methoddual0}, we obtain the SDA method:
\begin{equation}\label{eq:SDA-compact0}
\boxed{\quad y^{k+1} =  y^k + S \left(S^\top A B^{-1} A^\top S\right)^\dagger S^\top \left(b - A \left( c+B^{-1}A^\top y^k \right) \right) \quad }
\end{equation}

The symbol $\dagger$ denotes the Moore-Penrose pseudoinverse\footnote{It is known that the vector $M^\dagger d$ is the least-norm solution of the least-squares problem $\min_{\lambda} \|M \lambda - d\|^2$. Hence, if the system $M\lambda =d$ has a solution, then $M^\dagger d = \arg \min_\lambda \{\| \lambda \| \;:\; M \lambda =d\}$.}.

To the best of our knowledge, a randomized optimization algorithm  with iterates of the {\em general} form \eqref{eq:methoddual0}  was not considered nor analyzed before. In the special case when $S$ is chosen to be a random unit coordinate vector, SDA specializes to  the {\em randomized coordinate descent method}, first analyzed by Leventhal and Lewis \cite{Leventhal:2008:RMLC}. 
In the special case when $S$ is chosen as a random column submatrix of the $m\times m$ identity matrix, SDA specializes to the {\em randomized Newton method} of Qu, Fercoq, Richt\'{a}rik and Tak\'{a}\v{c} \cite{SDNA}.

With the dual iterates $\{y^k\}$ we associate a sequence of primal iterates $\{x^k\}$ as follows:
\begin{equation} \label{eq:primaliterates0} x^k \eqdef c + B^{-1}A^\top y^k.\end{equation}
In combination with \eqref{eq:SDA-compact0}, this yields the primal iterative process
 \begin{equation}\label{eq:SDA-primal0}
\boxed{\quad x^{k+1} =  x^k - B^{-1}A^\top S \left(S^\top A B^{-1} A^\top S\right)^\dagger S^\top \left(A x^k  -b \right) \quad }
\end{equation}

Optimality conditions (see Section~\ref{subsec:OptCond}) imply that if $y^*$ is any dual optimal point, then $c+B^{-1}A^\top y^*$ is necessarily primal optimal and hence equal to $x^*$, the optimal solution of \eqref{eq:P}. Moreover,  we have the following useful and insightful correspondence between the quality of the primal and dual iterates (see Proposition~\ref{lem:correspondence}):
\begin{equation}\label{eq:iugs8gs}D(y^*) - D(y^k) = \tfrac{1}{2}\|x^k-x^*\|_B^2.\end{equation}
Hence, {\em dual convergence in  function values is equivalent to primal convergence in iterates.}

Our work belongs to a growing literature on randomized methods for various problems appearing in linear algebra, optimization and computer science. In particular, relevant methods include sketching algorithms, randomized Kaczmarz, stochastic gradient descent and their variants  \cite{SV:Kaczmarz2009,Needell2010,Drineas2011,hogwild,Zouzias2012, Needell2012,Needell2012a, Ramdas2014, SAG, pegasos2,SVRG, NSync, S2GD, proxSVRG, SAGA, mS2GD, IProx-SDCA,Needell2014, Dai2014, NeedellWard2015, Ma2015a, Gower2015b,Oswald2015,LiuWright-AccKacz-2016}
and randomized coordinate and subspace type methods and their variants \cite{Leventhal:2008:RMLC,Lin:2008:DCDM,ShalevTewari09,Nesterov:2010RCDM,Wright:ABCRRO,shotgun,UCDC,nesterov2011random,PCDM,
tao2012stochastic,Necoara:Parallel,ICD,Necoara:rcdm-coupled,Hydra,Hydra2,SDCA,Fercoq-paralleladaboost,APPROX, Lee2013,QUARTZ,SPCDM,ALPHA,ESO,SPDC,SDNA,NIPSdistributedSDCA,ASDCA,
APCG,AdaSDCA,Gower2015b}.

\subsection{The main results} 

We now describe two complexity theorems which form the core theoretical contribution of this work. The results hold for a wide family of distributions ${\cal D}$, which we describe next.

\paragraph{Weak assumption on ${\cal D}$.} In our analysis, we only impose a very weak assumption on $\cal D$. In particular, we only assume that the $m\times m$ matrix \begin{equation} \label{eq:H} H \eqdef \mathbf{E}_{S\sim {\cal D}} \left[ S\left(S^\top AB^{-1}A^\top S\right)^{\dagger}S^\top\right]\end{equation}
 is well defined and nonsingular\footnote{It is known that the pseudoinverse of a symmetric positive semidefinite matrix is again symmetric and positive semidefinite. As a result, if the expectation defining $H$ is finite, $H$ is also symmetric and positive semidefinite. Hence, we could equivalently assume that $H$ be positive definite.
}.   Hence, we do not assume that $S$ be picked from any particular random matrix ensamble: the options are, quite literally, limitless. This makes it possible for practitioners to choose the best distribution specific to a particular application. 

We cast the first complexity result in terms of the primal iterates since solving \eqref{eq:P} is our main focus in this work. Let $\Range{M}, \Rank{M}$ and $\lambda_{\min}^+(M)$ denote the range space, rank and the smallest nonzero eigenvalue of $M$, respectively.

\begin{theorem}[\bf Convergence of primal iterates and of the residual] \label{theo:Enormerror}
Assume that the matrix $H$, defined in \eqref{eq:H}, is nonsingular. Fix arbitrary $x^0\in \R^n$. The primal iterates $\{x^k\}$ produced by \eqref{eq:SDA-primal0} converge exponentially fast in expectation to $x^* + t$, where  $x^*$ is the optimal solution of the primal problem \eqref{eq:P}, and $t$ is the projection of $x^0-c$ onto $\Null{A}$:
\begin{equation}\label{eq:def_of_t}t \eqdef \arg \min_{t'} \left\{ \|x^0-c - t'\|_B \;:\; t'\in \Null{A}\right\}.\end{equation}
In particular,  for all $k\geq 0$ we have
\begin{eqnarray}\label{eq:Enormerror} \text{Primal iterates:} \quad &&\E{\norm{x^{k} - x^{*}- t}_{B}^2}\leq \rho^k \cdot \norm{x^{0} -  x^{*} - t}_{B}^2,\\
\label{eq:s98h09hsxxx} \text{Residual:} \quad && \E{\|A x^k-b\|_B} \leq \rho^{k/2} \|A\|_B \|x^0-x^*-t\|_B + \|At\|_B,\end{eqnarray}
where  $\|A\|_B \eqdef \max \{\|Ax\|_B \;:\; \|x\|_B\leq 1\}$ and 
\begin{equation}\label{eq:rho} \rho \eqdef 1- \lambda_{\min}^+\left(B^{-1/2}A^\top H A B^{-1/2}\right).
\end{equation}
 Furthermore, the convergence rate is bounded by
\begin{equation} \label{eq:nubound}
1-\frac{\E{\Rank{S^\top A}}}{\Rank{A}}\leq \rho < 1.
\end{equation}

\end{theorem}

If we let $S$ be a unit coordinate vector chosen at random, $B$  be the identity matrix and set $c=0$, then \eqref{eq:SDA-primal0} reduces to the {\em randomized Kaczmarz (RK)} method proposed and analyzed in a seminal work of Strohmer and Vershynin \cite{SV:Kaczmarz2009}. Theorem~\ref{theo:Enormerror}
implies that RK converges with an exponential rate so long as the system matrix has no zero rows (see Section~\ref{sec:discrete}).  To the best of our knowledge, such a result was not previously established: current convergence results for RK assume that the system matrix is full rank~\cite{Ma2015a, Ramdas2014}. Not only do we show that the RK method converges to the least-norm solution for any consistent system, but we do so through a single all encompassing theorem covering a wide family of algorithms. Likewise, convergence of block variants of RK has only been established for full column rank~\cite{Needell2012,Needell2014}. Block versions of RK can be obtained from our generic method by choosing $B=I$ and $c=0$, as before, but letting $S$ to be a random column submatrix of the identity matrix (see \cite{Gower2015b}). Again, our general complexity bound holds under no assumptions on $A$, as long as one can find $S$ such that $H$ becomes nonsingular.

The lower bound \eqref{eq:nubound} says that for a singular system matrix, the number of steps required by SDA to reach an expected accuracy is at best inversely proportional to the rank of $A$. If $A$ has row rank equal to one, for instance, then  RK converges in one step (this is no surprise, given that RK projects onto the solution space of a single row, which in this case, is the solution space of the whole system). Our lower bound in this case becomes $0$, and hence is 
tight. 

While Theorem~\ref{theo:Enormerror} is cast in terms of the primal  iterates, if we  assume that $x^0 = c+ B^{-1}A^\top y^0$ for some $y^0\in \R^m$, then an equivalent dual characterization follows by combining  \eqref{eq:primaliterates0} and \eqref{eq:iugs8gs}. In fact, in that case we can  also establish the convergence of the primal function values and of the duality gap. {\em No such results were previously known. }

\begin{theorem}[\bf Convergence of function values]\label{theo:2}
Assume that the matrix $H$, defined in \eqref{eq:H}, is nonsingular. Fix arbitrary $y^0\in \R^m$ and let $\{y^k\}$ be the SDA iterates produced by \eqref{eq:SDA-compact0}. Further, let $\{x^k\}$ be the associated primal iterates, defined by \eqref{eq:primaliterates0}, $OPT  \eqdef P(x^*)=D(y^*)$,
\[U_0 \eqdef \tfrac{1}{2}\|x^0-x^*\|_B^2 \overset{\eqref{eq:iugs8gs}}{=} OPT -D(y^0),\] 
and let $\rho$ be as in Theorem~\ref{theo:Enormerror}. Then for all $k\geq 0$ we have the following complexity bounds:
\begin{eqnarray} 
\text{Dual suboptimality:} \quad &&\E{OPT - D(y^k)}\leq \rho^k U_0\label{eq:DUALSUBOPT} \\
\text{Primal suboptimality:} \quad &&\E{P(x^k) - OPT}\leq \rho^k U_0 +  2 \rho^{k/2}  \sqrt{OPT \times U_0} \label{eq:PRIMALSUBOPT} \\
 \text{Duality gap:} \quad && \E{P(x^k) - D(y^k)}\leq 2\rho^k U_0 +  2 \rho^{k/2} \sqrt{OPT\times U_0}  \label{eq:GAPSUBOPT}
\end{eqnarray}

\end{theorem}

Note that the dual objective function is {\em not} strongly concave  in general, and yet we prove linear convergence (see \eqref{eq:DUALSUBOPT}). It is known that for {\em some} structured optimization problems,  linear convergence results  can be obtained without the need to assume strong concavity (or strong convexity, for minimization problems). Typical approaches to such results  would be via the employment of error bounds \cite{LuoTseng93-AOR, Tseng95-JCAM, HongLuo2013,Ma-Tapp-Takac-FeasibleDescent-2015,NecoaraClipiciSIOPT2016}.
{\em In our analysis, no error bounds are necessary.}

\subsection{Outline}

The paper is structured as follows.  Section~\ref{sec:SDA}  describes the algorithm in detail, both in its dual and primal form, and establishes several useful identities. In Section~\ref{sec:discrete} we characterize discrete distributions for which our main assumption on $H$ is satisfied. We then specialize our method to several simple discrete distributions to better illustrate the results. We then show in Section~\ref{sec:gossip} how SDA can be applied to design new randomized gossip algorithms. We also show that our framework can recover some standard methods. Theorem~\ref{theo:Enormerror} is proved in Section~\ref{sec:proof} and Theorem~\ref{theo:2} is proved in Section~\ref{sec:proof2}. In Section~\ref{sec:experiments} we perform a simple experiment illustrating the convergence of the randomized Kaczmarz method on rank deficient linear systems. We conclude in Section~\ref{sec:conclusion}. To the appendix  we relegate two elementary but useful technical results which are needed multiple times in the text.

\section{Stochastic Dual Ascent} \label{sec:SDA}

By {\em stochastic  dual ascent} (SDA) we refer to a randomized optimization method for solving the dual problem \eqref{eq:D} performing iterations of the form
\begin{equation}\label{eq:methoddual} y^{k+1} = y^k + S \lambda^k,\end{equation}
where $S$ is a random matrix with $m$ rows drawn in each iteration independently from a prespecified distribution. We shall not fix the number of columns of $S$; in fact, we even allow for the number of columns to be random. By performing steps of the form \eqref{eq:methoddual}, we are moving in the range space of the random matrix $S$, with $\lambda^k$ describing the precise linear combination of the columns used in computing the step. 
In particular, we shall choose $\lambda^k$ from the set 
\[Q^k \eqdef \arg \max_{\lambda} D(y^k + S\lambda) \overset{\eqref{eq:D}}{=} \arg\max_{\lambda} \left\{ (b-Ac)^\top (y^k + S\lambda) - \tfrac{1}{2}\left\|A^\top (y^k + S \lambda)\right\|_{B^{-1}}^2\right\}.\]
Since $D$ is bounded above (a consequence of weak duality), this set is nonempty. Since $D$ is a concave quadratic, $Q^k$ consists of all those vectors $\lambda $ for which the gradient of the mapping $\phi_k(\lambda): \lambda \mapsto D(y^k + S\lambda)$ vanishes. 
This leads to the observation that $Q^k$ is the set of solutions of a random linear system:
\[Q^k =  \left\{\lambda \in \R^m \;:\; \left(S^\top A B^{-1}A^\top S \right) \lambda = S^\top \left(b - Ac - A B^{-1}A^\top y^k \right) \right\}.\]
If $S$ has a small number of columns, this is a small easy-to-solve system. 

A key feature of our method enabling us to prove exponential error decay despite the lack of strong concavity is the  way in which we choose $\lambda^k$ from $Q^k$. In  SDA, $\lambda^k$  is chosen to be the least-norm  element of $Q^k$,
\[\lambda^k \eqdef \arg\min_{\lambda \in Q^k} \|\lambda\|,\]
where $\|\lambda\| = (\sum_i \lambda_i^2)^{1/2}$ denotes standard Euclidean norm. The least-norm solution of a linear system can be written down in a compact way using the (Moore-Penrose) pseudoinverse. In our case, we obtain the formula
\begin{equation}\label{eq:lambda_closed_form}
\lambda^k = \left(S^\top A B^{-1} A^\top S\right)^\dagger S^\top \left(b - Ac - A B^{-1}A^\top y^k  \right),
\end{equation}
where $\dagger$ denotes the pseudoinverse operator.
Note that if $S$ has only a few columns, then~\eqref{eq:lambda_closed_form} requires projecting the origin onto a small linear system. 
The SDA algorithm is obtained by combining \eqref{eq:methoddual} with \eqref{eq:lambda_closed_form}.

\begin{algorithm}[!h]
\begin{algorithmic}[1]
\State \textbf{parameter:} ${\cal D}$ = distribution over random matrices
\State Choose $y^0 \in \R^m$
\Comment Initialization
\For {$k = 0, 1, 2, \dots$}
	\State Sample an independent copy $S\sim {\cal D}$
	\State $\lambda^k = \left(S^\top A B^{-1} A^\top S\right)^\dagger S^\top \left(b - A c - AB^{-1}A^\top y^k  \right)$
	\State $y^{k+1} =  y^k + S  \lambda^k$
		\Comment Update the dual variable
\EndFor
\end{algorithmic}

\caption{Stochastic Dual Ascent (SDA)}
\label{alg:SDA}
\end{algorithm}

The method has one parameter: the distribution $\cal D$ from which the random matrices $S$ are drawn. Sometimes, one is interested in finding any solution of the system $Ax=b$, rather than the particular solution described by the primal problem 
\eqref{eq:P}. In such situations, $B$ and $c$ could also be seen as parameters.

\subsection{Optimality conditions} \label{subsec:OptCond}

For any $x$ for which $Ax=b$ and for any $y$ we have
\[P(x) - D(y) \overset{\eqref{eq:P}+\eqref{eq:D}}{=} \tfrac{1}{2}\|x-c\|_B^2 + \tfrac{1}{2}\|A^\top y\|_{B^{-1}}^2 + (c-x)^\top A^\top y \geq 0,\]
where the inequality (weak duality) follows from the Fenchel-Young inequality\footnote{Let $U$ be a vector space equipped with an inner product $\langle \cdot, \cdot \rangle : U\times U \to \R$. Given a function $f:U \to \R$, its convex (or Fenchel) conjugate $f^*:U\to \R\cup \{+\infty\}$ is defined by $f^*(v) = \sup_{u \in U} \langle u, v \rangle - f(u)$. A direct consequence of this is the  Fenchel-Young inequality, which asserts that  $f(u) + f^*(v)\geq \langle u, v\rangle$ for all $u$ and $v$.  The inequality in the main text follows by choosing $f(u)=\tfrac{1}{2}\|u\|_B^2$ (and hence  $f^*(v)=\tfrac{1}{2}\|v\|^2_{B^{-1}}$), $u=x-c$ and $v=A^\top y$. If $f$ is differentiable, then equality holds if and only if $v=\nabla f(u)$. In our case, this condition is $x=c+B^{-1}A^\top y$. This, together with primal feasibility, gives  the optimality conditions \eqref{eq:opt_cond}. For more details on Fenchel duality, see \cite{bookBorweinLewis2006}.}. As a result, we obtain the following necessary and sufficient optimality conditions, characterizing primal and dual optimal points.

\begin{proposition} [Optimality conditions] \label{eq:prop_opt_cond}Vectors $x\in \R^n$ and $y\in\R^m$ are optimal for the primal \eqref{eq:P} and dual \eqref{eq:D} problems respectively, if and only if they satisfy the following relation
\begin{equation} \label{eq:opt_cond}Ax = b, \qquad x = c + B^{-1} A^\top y.\end{equation}
\end{proposition}

In  view of this, it will be useful to define a linear mapping from $\R^m$ to $\R^n$ as follows:
\begin{equation}\label{eq:98s98hs}x(y) = c + B^{-1}A^\top y.\end{equation}

As an immediate corollary of Proposition~\ref{eq:prop_opt_cond} we observe that for any dual optimal $y^*$, the vector $x(y^*)$ must be primal optimal. Since the primal problem has a unique optimal solution, $x^*$,   we must necessarily have \begin{equation}\label{eq:opt_primal}x^* =x(y^*) = c + B^{-1} A^\top y^*.\end{equation}

Another immediate corollary of Proposition~\ref{eq:prop_opt_cond} is the following characterization of dual optimality:  $y$ is dual optimal if and only if
\begin{equation} \label{eq:98hs8h9sss}b  - Ac = AB^{-1}A^\top y.\end{equation}
Hence, the set of dual optimal solutions is
${\cal Y}^* = (AB^{-1}A^\top)^\dagger (b-Ac) + \Null{AB^{-1}A^\top}$. Since, $\Null{AB^{-1}A^\top} = \Null{A^\top}$ (see Lemma~\ref{lem:09709s}), we have
\[{\cal Y}^* = \left(AB^{-1}A^\top\right)^\dagger (b-Ac) + \Null{A^\top}.\]

Combining this with \eqref{eq:opt_primal}, we get

\[x^* = c + B^{-1}A^\top \left(AB^{-1}A^\top \right)^\dagger(b-Ac).\]

\begin{remark} [The dual is also a least-norm problem.] Observe that:
\begin{enumerate}
\item 
The particular dual optimal point $y^* = (AB^{-1}A^\top)^\dagger (b-Ac)$ is the solution of the following optimization problem:

\begin{equation} \label{eq:iusiuh7ss}\min \left\{ \tfrac{1}{2}\|y\|^2 \;:\; A B^{-1} A^\top y = b-Ac\right\}.\end{equation}

Hence, this particular formulation of the dual problem has the same form as the primal problem: projection onto a linear system.

\item If $A^\top A$ is positive definite (which can only happen if $A$ is of full column rank, which means that $Ax=b$ has a unique solution and hence the primal objective function does not matter), and we choose $B=A^\top A$, then the dual constraint \eqref{eq:iusiuh7ss} becomes

\[A (A^\top A)^{-1}A^\top y = b - Ac.\]

This constraint has a geometric interpretation: we are seeking vector $y$ whose orthogonal projection onto the column space of $A$ is equal to $b-Ac$. Hence the reformulated dual problem \eqref{eq:iusiuh7ss} is asking us to find the vector $y$ with this property having the least norm.

\end{enumerate}
\end{remark}

\subsection{Primal iterates associated with the dual iterates}

With the sequence  of dual  iterates $\{y^k\}$ produced by SDA  we can associate a sequence of primal iterates $\{x^k\}$  using the mapping \eqref{eq:98s98hs}:
\begin{equation} \label{eq:primaliterates} x^k \eqdef x(y^k) = c + B^{-1}A^\top y^k.\end{equation}

This leads to the following {\em primal version of the SDA method}.

\begin{algorithm}[!h]
\begin{algorithmic}[1]
\State \textbf{parameter:} ${\cal D}$ = distribution over random matrices
\State Choose $x^0 \in \R^n$
\Comment Initialization
\For {$k = 0, 1, 2, \dots$}
	\State Sample an independent copy $S\sim {\cal D}$
	\State $x^{k+1} =  x^k - B^{-1}A^\top S \left(S^\top A B^{-1} A^\top S\right)^\dagger S^\top (A x^k  -b )$
		\Comment Update the primal variable
\EndFor
\end{algorithmic}

\caption{Primal Version of Stochastic Dual Ascent (SDA-Primal)}
\label{alg:SDA-Primal}
\end{algorithm}

\begin{remark} \label{lem:5shsuss} 

A couple of observations:

\begin{enumerate}

\item {\em Self-duality.} If $A$ is positive definite, $c=0$, and if we choose $B=A$, then in view of \eqref{eq:primaliterates}  we have $x^k = y^k$ for all $k$, and hence Algorithms~\ref{alg:SDA} and \ref{alg:SDA-Primal} coincide. In this case, Algorithm~\ref{alg:SDA-Primal} can be described as {\em self-dual.}

\item {\em Space of iterates.} A direct consequence of the correspondence between the dual and primal iterates \eqref{eq:primaliterates}  is the following simple observation (a generalized version of this, which we prove later as Lemma~\ref{lem:error}, will be used  in the proof of Theorem~\ref{theo:Enormerror}): Choose $y^0\in \R^m$ and let $x^0 = c + B^{-1}A^\top y^0$. Then the iterates $\{x^k\}$ of Algorithm~\ref{alg:SDA-Primal} are of the form $x^k = c + B^{-1} A^\top y^k$ for some $y^k\in \R^m$.

\item {\em Starting point.} While we have defined the primal iterates of Algorithm~\ref{alg:SDA-Primal}  via a linear transformation of the dual iterates---see \eqref{eq:primaliterates}---we {\em can}, in principle,  choose $x^0$ arbitrarily, thus breaking the primal-dual connection which helped us to define the method. In particular, we can choose $x^0$ in such a way that there does not exist $y^0$ for which $x^0 = c + B^{-1}A^\top y^0$. As is clear from Theorem~\ref{theo:Enormerror}, in this case the iterates $\{x^k\}$ will not converge to $x^*$, but to $x^*+t$, where $t$ is the projection of $x^0-c$ onto the nullspace of $A$.
\end{enumerate}

\end{remark}

It turns out that Algorithm~\ref{alg:SDA-Primal} is equivalent to the {\em sketch-and-project} method \eqref{eq:sketchandproject} of Gower and Richt\'{a}rik~\cite{Gower2015b}:
\begin{equation}\label{eq:sketchandproject}x^{k+1} = \arg \min_{x} \left\{ \|x-x^k\|_B \;:\; S^\top Ax = S^\top b \right\},\end{equation}
where $S$ is a random matrix drawn in an i.i.d.\ fashion from a fixed distribution, just as in this work. 
In this method, the ``complicated'' system $Ax=b$ is first replaced by its sketched version $S^\top Ax =S^\top b$, the solution space of which contains all solutions of the original system. 
If $S$ has a few columns only, this system will be small and easy to solve. Then, progress is made by projecting the last iterate onto the sketched system. 

We now briefly comment on the relationship between \cite{Gower2015b} and our work.

\begin{itemize}

\item \textbf{Dual nature of sketch-and-project.} It was shown in \cite{Gower2015b} that Algorithm~\ref{alg:SDA-Primal} is equivalent to the sketch-and-project method. In fact, the authors of \cite{Gower2015b} provide five additional equivalent formulations of sketch-and-project, with  Algorithm~\ref{alg:SDA-Primal} being one of them. Here we show that their method can be seen as a primal process associated with SDA, which is a new method operating in the dual. By observing this,  we uncover a hidden dual nature of the sketch-and-project method.  For instance, this allows us to formulate and prove Theorem~\ref{theo:2}. No such results appear in \cite{Gower2015b}.

\item \textbf{No assumptions on the system matrix.} 
 In \cite{Gower2015b}  the authors only studied the convergence of the primal iterates $\{x^k\}$, establishing a (much) weaker variant of Theorem~\ref{theo:Enormerror}. Indeed,  convergence was only established in the case when $A$ has full column rank. In this work, we lift this assumption completely and hence establish complexity results in the  general case. 

\item \textbf{Convergence to a shifted point.} As we show in Theorem~\ref{theo:Enormerror},  Algorithm~\ref{alg:SDA-Primal}  converges to $x^*+t$, where $t$ is the projection of $x^0-c$ onto $\Null{A}$. Hence, in general, the method does not converge to the optimal solution $x^*$. This is not an issue if $A$ is of full column rank---an assumption used in the analysis in \cite{Gower2015b}---since then $\Null{A}$ is trivial and hence $t=0$. As long as $x^0-c$ lies in $\Range{B^{-1}A^\top}$, however, we have $x^k\to x^*$. This can be easily enforced (for instance, we can choose $x^0=c$).

\end{itemize}

\subsection{Relating the quality of the dual and primal iterates}

The following simple but insightful result (mentioned in the introduction) relates the ``quality'' of a dual vector $y$ with that of its primal counterpart, $x(y)$. It says that the dual suboptimality of $y$ in terms of function values is equal to the primal suboptimality of $x(y)$ in terms of distance.
 
\begin{proposition}\label{lem:correspondence} Let $y^*$ be any dual optimal point and $y\in \R^m$. Then
\[D(y^*) - D(y) = \tfrac{1}{2}\|x(y^*) - x(y)\|_B^2.\]
\end{proposition}
\begin{proof}
Straightforward calculation shows that
\begin{eqnarray*}
D(y^*)  -D(y) &\overset{\eqref{eq:D}}{=}& (b-Ac)^\top (y^* - y) - \tfrac{1}{2}(y^*)^\top A B^{-1} A^\top y^*  + \tfrac{1}{2}y^\top A B^{-1} A^\top y\\
&\overset{\eqref{eq:98hs8h9sss}}{=}&(y^*)^\top A B^{-1} A^\top (y^* - y) - \tfrac{1}{2}(y^*)^\top A B^{-1} A^\top y^*  + \tfrac{1}{2}y^\top A B^{-1} A^\top y\\
&=& \tfrac{1}{2}(y-y^*)^\top AB^{-1} A^\top (y-y^*)\\
&\overset{\eqref{eq:98s98hs}}{=}& \tfrac{1}{2}\|x(y) - x(y^*)\|_B^2. 
\end{eqnarray*}
\qed
\end{proof}

Applying this result to sequence $\{(x^k,y^k)\}$ of dual iterates produced by SDA and their corresponding primal images, as defined in \eqref{eq:primaliterates}, we get the identity:
\[D(y^*)- D(y^k) = \tfrac{1}{2}\|x^k - x^*\|_B^2.\]
Therefore, {\em  dual convergence in function values $D(y^k)$  is equivalent to primal convergence in iterates $x^k$}. Furthermore, a direct computation leads to the following formula for the {\em duality gap}:
\begin{equation}\label{eq:dualitygap09709709}P(x^k ) - D(y^k) \overset{\eqref{eq:primaliterates}}{=} (AB^{-1}A^\top y^k + Ac - b)^\top y^k = -(\nabla D(y^k) )^\top y^k.\end{equation}

Note that computing the gap is significantly more expensive than the cost of a single iteration (in the interesting regime  when the number of columns of $S$ is small). Hence, evaluation of the duality gap should generally be avoided. If it is necessary to be certain about the quality of a solution however,  the above formula will be useful. The gap should then be computed from time to time only, so that this extra work does not significantly slow down the iterative process.

\section{Discrete Distributions} \label{sec:discrete}

Both the SDA algorithm and its primal counterpart are generic in the sense that the distribution  $\cal D$ is not specified beyond assuming that the matrix $H$ defined in \eqref{eq:H} is well defined and nonsingular. In this section we shall first characterize finite discrete distributions for which $H$ is nonsingular. We  then give a few examples of algorithms based on such distributions, and comment on our complexity results in more detail.

\subsection{Nonsingularity of $H$ for finite discrete distributions}

For simplicity, we shall focus on {\em finite discrete} distributions $\cal D$. That is, we set $S = S_i$ with probability $p_i>0$, where  $S_1,\dots,S_r$ are fixed matrices (each with $m$ rows).  The next theorem gives a necessary and sufficient condition for the matrix $H$ defined in \eqref{eq:H} to be nonsingular.
 
\begin{theorem}\label{thm:H} Let $\cal{D}$ be a finite discrete distribution, as described above. Then  $H$ is nonsingular if and only if
\[\Range{[S_1S_1^\top A, \cdots, S_r S_r^\top A]} = \R^m .\]
\end{theorem}

\begin{proof} Let $K_i = S_i^\top AB^{-1/2}$. In view of the identity $\left(K_i K_i^\top \right)^{\dagger} = (K_i^\dagger )^\top K_i^\dagger$,  we can write 
\[H \overset{\eqref{eq:H}}{=} \sum_{i=1}^r H_i,\]
where $H_i = p_i S_i (K_i^\dagger)^\top K_i^\dagger S_i^\top$. Since $H_i$ are symmetric positive semidefinite,  so is $H$.  Now,  it is easy to check that $y^\top H_i y = 0$ if and only if $y  \in \Null{H_i}$ (this holds for any symmetric positive semidefinite $H_i$). Hence, $y^\top H y = 0$ if and only if $y \in \cap_i \Null{H_i}$ and thus
$H$ is positive definite if and only if 
\begin{equation} \label{eq:0h09sh0976}\bigcap_{i} \Null{H_i} = \{0\}.\end{equation}
In view of Lemma~\ref{lem:09709s}, $\Null{H_i} = \Null{\sqrt{p_i}K_i^\dagger S_i^\top} =  \Null{K_i^\dagger S_i^\top}$. Now, $y\in \Null{K_i^\dagger S_i^\top}$ if and only of $S_i^\top y \in \Null{K_i^\dagger} = \Null{K_i^\top} = \Null{A^\top S_i}$. Hence, $\Null{H_i} = \Null{A^\top S_i S_i^\top}$, which means that \eqref{eq:0h09sh0976} is equivalent to $\Null{[S_1S_1^\top A, \cdots, S_r S_r^\top A]^\top} = \{0\}$. \qed
\end{proof}

\bigskip
We have the following corollary.\footnote{We can also prove the corollary directly as follows: The first assumption implies that $S_i^\top A B^{-1} A^\top S_i$ is invertible for all $i$ and that $V \eqdef \mbox{Diag}\left(p_i^{1/2}(S_i^\top A{B^{-1}}A^\top S_i)^{-1/2}\right)$  is nonsingular. It remains to note that
\[
H \overset{\eqref{eq:H}}{=}  \E{ S\left(S^\top AB^{-1}A^\top S\right)^{-1} S^\top} \\
= \sum_i p_i  S_i\left(S_i^\top AB^{-1}A^\top S_i \right)^{-1} S_i^\top  = \bar{S}V^2 \bar{S}^\top.
\]
}

\begin{corollary}\label{cor:09hs09hs}
Assume that $S_i^\top A$ has full row rank for all $i$ and that  $\bar{S} \eqdef [S_1,\ldots, S_r]$ is of full row rank. Then $H$ is nonsingular. 
\end{corollary}

We now give a few illustrative examples:

\begin{enumerate}
\item \emph{Coordinate vectors.} Let $S_i = e_i$ ($i^{\text{th}}$ unit coordinate vector) for $i=1,2,\dots,r=m$. In this case,  $\bar{S} = [S_1,\dots,S_m]$ is the identity matrix in $\R^m$, and $S_i^\top A$ has full row rank for all $i$ as long as the rows of $A$ are all nonzero. By Corollary~\ref{cor:09hs09hs}, $H$ is positive definite.
\item \emph{Submatrices of the identity matrix.} We can let $S$ be a random column submatrix of the $m\times m$ identity matrix $I$. There are  $2^m-1$ such potential submatrices, and  we choose $1\leq r \leq 2^m-1$.  As long as we choose $S_1,\dots,S_r$ in such a way that each column of $I$ is represented in some matrix $S_i$, the matrix $\bar{S}$ will have full row rank.  Furthermore, if  $S_i^\top A$ has full row rank for all $i$, then by the above corollary, $H$ is nonsingular. Note that if the row rank of $A$ is $r$, then the matrices $S_i$ selected by the  above process will  necessarily have at most $r$ columns.

\item \emph{Count sketch and Count-min sketch.} Many other ``sketching'' matrices $S$ can be employed within SDA, including the count sketch \cite{CountSketch2002} and the count-min sketch \cite{CountMinSketch2005}. In our context (recall that we sketch with the transpose of $S$), $S$ is a count-sketch matrix (resp. count-min sketch) if it is assembled from random columns of $[I,-I]$ (resp $I$), chosen uniformly with replacement, where $I$ is the $m\times m$ identity matrix.

\end{enumerate}

\subsection{Randomized Kaczmarz as the primal process associated with  randomized coordinate ascent } \label{subsec:RKvsRCA}

Let $B=I$ (the identity matrix). The primal problem then becomes
\begin{eqnarray}  \text{minimize} && P(x)\eqdef \tfrac{1}{2}\|x-c\|^2 \notag\\
\text{subject to} && Ax=b \notag\\
&& x\in \R^n.\notag
\end{eqnarray}
and the dual problem is
\begin{eqnarray}\notag \text{maximize} && D(y)\eqdef (b-Ac)^\top y - \tfrac{1}{2}y^\top A A^\top y\\
\text{subject to}&& y \in \R^m. \notag
\end{eqnarray}

\paragraph{Dual iterates.} Let us choose $S=e^i$ (unit coordinate vector in $\R^m$) with probability $p_i>0$ (to be specified later). The SDA method (Algorithm~\ref{alg:SDA}) then takes the form
\begin{equation}\label{eq:SDA-compact08986986098}
\boxed{\quad y^{k+1} =  y^k + \frac{b_i - A_{i}c - A_{i:} A^\top y^k }{\|A_{i:}\|^2}e_i   \quad }
\end{equation}
This is the randomized coordinate ascent method applied to the dual problem. In the form popularized by Nesterov  \cite{Nesterov:2010RCDM}, it takes the form
\[y^{k+1} = y^k + \frac{e_i^\top \nabla D(y^k)}{L_i} e_i,\]
where $e_i^\top \nabla D(y^k)$ is the $i$th partial derivative of $D$ at $y^k$ and $L_i>0$ is the Lipschitz constant of the $i$th partial derivative, i.e., constant for which the following inequality holds for all $\lambda\in \R$:
\begin{equation}\label{eq:s97g98gs} | e_i^\top \nabla D(y + \lambda e_i) - e_i^\top \nabla D(y)  | \leq L_i |\lambda|.\end{equation}
It can be easily verified that \eqref{eq:s97g98gs}  holds with $L_i=\|A_{i:}\|^2$ and that $e_i^\top \nabla D(y^k)=b_i - A_{i:}c - A_{i:} A^\top y^k $.

\paragraph{Primal iterates.} The associated primal iterative process (Algorithm~\ref{alg:SDA-Primal})  takes the form
\begin{equation}\label{eq:SDA-primal009s09us0098}
\boxed{\quad x^{k+1} =  x^k - \frac{A_{i:} x^k   -b_i }{\|A_{i:}\|^2} A_{i:}^\top  \quad }
\end{equation}
This is the randomized Kaczmarz method of Strohmer and Vershynin \cite{SV:Kaczmarz2009}.

\paragraph{The rate.}

\bigskip
Let us now compute the rate $\rho$ as defined in \eqref{eq:rho}.
It will be convenient, but {\em not} optimal, to choose the probabilities via
 \begin{equation}\label{eq:089h08hs98xx}p_i = \frac{\norm{ A_{i:} }_2^2}{\norm{A}_F^2},\end{equation}  where $\|\cdot\|_F$ denotes the Frobenius norm (we assume that $A$ does not contain any zero rows).  Since
\[H \overset{\eqref{eq:H}}{=} \E{S \left(S^\top A A^\top S \right)^{\dagger}S^\top } =  \sum_{i=1}^m p_i \frac{e_i e_i^\top }{\|A_{i:}\|^2} \overset{\eqref{eq:089h08hs98xx}}{=} \frac{1}{\norm{A}_F^2}I,\]
we have
\begin{equation}\label{eq:98hs8h8ss}\rho = 
 1-\lambda_{\min}^+\left(A^\top H A \right) = 1-\frac{\lambda_{\min}^+\left(A^\top A\right)}{\norm{A}_F^2}.
\end{equation}
In general, the rate $\rho$ is a function of the probabilities $p_i$. The inverse problem: ``How to set the probabilities so that the rate is optimized?'' is difficult. If $A$ is of full column rank, however, it leads to a semidefinite program \cite{Gower2015b}.

Furthermore, if $r= \Rank{A}$, then in view of \eqref{eq:nubound}, the rate  is bounded  as
\[ 1- \frac{1}{r}\leq  \rho <1. \]
Assume that $A$ is of rank $r=1$ and let  $A= uv^\top$. Then $A^\top A = (u^\top  u) v  v^\top$, and hence this matrix is also of rank 1. Therefore, $A^\top A$ has a single nonzero eigenvalue, which is equal its trace. Therefore, $\lambda_{\min}^+(A^\top A) = \Tr{A^\top A} = \|A\|^2_F$ and hence $\rho =0$. Note that the  rate $\rho$ reaches its lower bound and the method converges in one step.

\paragraph{Remarks.} For randomized coordinate ascent applied to (non-strongly) concave quadratics, rate \eqref{eq:98hs8h8ss} has been established by Leventhal and Lewis \cite{Leventhal:2008:RMLC}. However, to the best of our knowledge, this is the first time this rate has also been established for the randomized Kaczmarz method. We do not only prove this, but show that this is because the iterates of the two methods are linked via a linear relationship. In the $c=0, B=I$ case, and for row-normalized matrix $A$, this linear relationship between the two methods was  recently independently observed by Wright~\cite{Wright:CoorDescMethods-survey}.  While all linear complexity results for  RK we are aware of require full rank assumptions, there exist nonstandard variants of RK which do not require such assumptions, one example being the asynchronous parallel version of RK studied by Liu, Wright and Sridhar \cite{Wright:AsyncPRK}.  Finally, no results of the type \eqref{eq:PRIMALSUBOPT} (primal suboptimality)  and \eqref{eq:GAPSUBOPT} (duality gap) previously existed for these methods in the literature. 

\subsection{Randomized block Kaczmarz is the primal process associated with randomized Newton}

Let $B=I$, so that we have the same pair of primal dual problems as in Section~\ref{subsec:RKvsRCA}.

\paragraph{Dual iterates.} Let us now choose $S$ to be a random column submatrix of the $m\times m$ identity matrix $I$. That is, we choose a  random subset $C\subset \{1,2,\dots,m\}$ and then let $S$ be the concatenation of columns $j\in C$ of $I$. We shall write $S=I_C$. Let $p_C$ be the probability that $S = I_C$. Assume that for each $j \in \{1,\dots,m\}$ there exists $C$ with $j \in C$ such that $p_C>0$. Such a random set is called {\em proper}~\cite{SDNA}.  

The SDA method (Algorithm~\ref{alg:SDA}) then takes the form
\begin{equation}\label{eq:SDA-compact08986986098BLOCK}
\boxed{\quad y^{k+1} =  y^k + I_C \lambda^k  \quad }
\end{equation}
where $\lambda^k$ is chosen so that the dual objective is maximized (see \eqref{eq:lambda_closed_form}). This is a variant of the {\em randomized Newton method} studied in \cite{SDNA}. By examining \eqref{eq:lambda_closed_form}, we see that this method works by ``inverting'' randomized submatrices of the ``Hessian'' $AA^\top$. Indeed, $\lambda^k$ is in each iteration computed by solving a system with the matrix $I_C^\top A A^\top I_C$. This is the random submatrix of $A A^\top$ corresponding to rows and columns in $C$.

\paragraph{Primal iterates.} In view of the equivalence between Algorithm~\ref{alg:SDA-Primal} and the sketch-and-project method \eqref{eq:sketchandproject}, the  primal iterative process associated with the randomized Newton method has the form
\begin{equation}\label{eq:SDA-primal009s09us0098BLOCK}
\boxed{\quad x^{k+1} = \arg \min_{x} \left\{ \|x-x^k\| \;:\; I_C^\top Ax = I_C^\top b \right\} \quad }
\end{equation}
This method is a variant of the {\em randomized block Kaczmarz} method of Needell \cite{Needell2012}. The method proceeds by projecting the last iterate $x^k$ onto a subsystem of $Ax=b$ formed by equations indexed by the set $C$.

\paragraph{The rate.} Provided that $H$ is nonsingular, the shared  rate of the randomized Newton and randomized block Kaczmarz methods is
\[\rho   = 1- \lambda_{\min}^+\left(A^\top \E{I_C\left(I_C^\top AA^\top I_C\right)^\dagger I_C^\top} A\right).\]

Qu et al \cite{SDNA} study the randomized Newton method for the problem of minimizing a smooth strongly convex function and prove linear convergence. In particular, they study the above rate in the case when $AA^\top$ is positive definite. Here we show that linear converges also holds for {\em weakly} convex quadratics (as long as $H$ is nonsingular). 

An interesting feature of the randomized Newton method, established in \cite{SDNA}, is  that when  viewed as a family of methods indexed by the size $\tau=|C|$, it enjoys superlinear speedup in $\tau$. That is, as $\tau$ increases by some factor, the iteration complexity drops by a factor that is at least as large.  It is possible to conduct a similar study in our setting  with a possibly singular matrix $AA^\top$, but such a study is not trivial and we therefore leave it for future research.

\subsection{Self-duality  for positive definite $A$}

If $A$ is positive definite, then we can choose $B=A$. As mentioned before, in this setting SDA is self-dual: $x^k=y^k$ for all $k$. The primal problem then becomes
\begin{eqnarray}  \text{minimize} && P(x)\eqdef \tfrac{1}{2}x^\top A x \notag\\
\text{subject to} && Ax=b \notag\\
&& x\in \R^n.\notag
\end{eqnarray}
and the dual problem becomes
\begin{eqnarray}\notag \text{maximize} && D(y)\eqdef b^\top y - \tfrac{1}{2}y^\top A  y\\
\text{subject to}&& y \in \R^m. \notag
\end{eqnarray}

Note that the primal objective function  does not play any role in determining the solution; indeed, the feasible set contains a single point only: $A^{-1}b$. However, it does affect the iterative process.
 
\paragraph{Primal and dual iterates.} As before, let us choose $S=e^i$ (unit coordinate vector in $\R^m$) with probability $p_i>0$, where the probabilities $p_i$ are arbitrary. Then both the  primal and the dual iterates take the form 
\[
\boxed{\quad y^{k+1} =  y^k - \frac{A_{i:}  y^k - b_i}{A_{ii}}e_i   \quad }
\]
This is the randomized coordinate ascent method applied to the dual problem. 

\paragraph{The rate.}

If we choose $p_i = A_{ii}/\Tr{A}$, then
\[H = \E{S\left(S^\top A S\right)^\dagger S^\top} = \frac{I}{\Tr{A}},\]
whence
\[\rho \overset{\eqref{eq:rho} }{=} 1 - \lambda_{\min}^+ \left(   A^{1/2}H A^{1/2}\right)  = 1 - \frac{ \lambda_{\min}(A)}{\Tr{A}}.\]

It is known that for this problem,  randomized coordinate descent applied to the dual problem, with this choice of probabilities, converges with this rate \cite{Leventhal:2008:RMLC}.

\section{Application: Randomized Gossip Algorithms} \label{sec:gossip}

In this section we apply our method and results to the distributed consensus (averaging) problem.

Let $(V,E)$ be a connected network with $|V|=n$ nodes and $|E|=m$ edges, where each edge is an unordered pair $\{i,j\} \in E$ of distinct nodes. Node $i \in V$ stores a private value $c_i\in \R$. The goal of a ``distributed consensus problem'' is for the network to compute the average of these private values in a distributed fashion \cite{RandGossip2006,OlshevskyTsitsiklis2009}. This means that the exchange of information  can only occur along the edges of the network.

The nodes may represent people in a social network, with edges representing friendship and private value representing certain private information, such as salary. The goal would be to compute the average salary via an iterative process where only friends are allowed to exchange information. The nodes may represent sensors in a wireless sensor network, with an edge between two sensors if they are close to each other so that they can communicate. Private values  represent measurements of some quantity performed by the sensors, such as the temperature. The goal is for the network to compute the average temperature.

\subsection{Consensus as a projection problem}

We now show how one can model the consensus (averaging) problem in the form \eqref{eq:P}. Consider the  projection problem
\begin{eqnarray}  \text{minimize} && \tfrac{1}{2}\|x - c\|^2 \notag \\
 \text{subject to} & & x_1=x_2=\cdots = x_n, \label{eq:ohs09hud98yd}
\end{eqnarray}
and note that the optimal solution $x^*$ must necessarily satisfy \[x^*_i=\bar{c}\eqdef \frac{1}{n}\sum_{i=1}^n c_i,\] for all $i$. There are many ways in which the constraint forcing all coordinates of $x$ to be equal can be represented in the form of a linear system $Ax=b$.   Here are some examples:

\begin{enumerate}
\item {\em Each node is equal to all its neighbours.} Let the equations of the system $Ax=b$ correspond to constraints \[x_i=x_j,\] for $\{i,j\}\in E$. That is, we are enforcing all pairs of vertices joined by an edge to have the same value.  Each edge $e\in E$ can be written in two ways: $e = \{i,j\}$ and $e=\{j,i\}$, where $i,j$ are the incident vertices. In order to avoid duplicating constraints,  for each edge $e\in E$ we use $e=(i,j)$ to denote an arbitrary but fixed order of its incident vertices $i,j$. We then let $A\in \R^{m\times n}$ and $b=0\in \R^m$, where  \begin{equation}\label{eq:89hs87s8ys}(A_{e:})^\top = f_i - f_j,\end{equation} and where  $e=(i,j)\in E$, $f_i$ (resp.\ $f_j$) is the $i^{\text{th}}$ (resp.\ $j^{\text{th}}$) unit coordinate vector in $\R^n$.  Note that the constraint $x_i=x_j$ is represented only once in the linear system. Further, note that the matrix \begin{equation}\label{eq:Laplacian09709}L = A^\top A \end{equation} is the {\em Laplacian} matrix of the graph $(V,E)$:
 \[L_{ij} = \begin{cases}
 d_i & i=j\\
 -1 & i\neq j,\,\, (i,j)\in E\\
 0 & \text{otherwise,}
 \end{cases}\]
 where $d_i$ is the degree of node $i$.

\item {\em Each node is the average of its neighbours.}  Let the equations of the system $Ax=b$ correspond to constraints 
\[x_i = \frac{1}{d_i}\sum_{j \in N(i)} x_j,\]
for $i\in V$, where $N(i)\eqdef \left\{ j\in V \,\,: \,\, \{i,j\} \in E\right\}$ is the set of neighbours of node $i$ and $d_i \eqdef |N(i)|$ is the degree of node $i$.  That is, we require that the values stored at each node are equal to the average of the values of its neighbours. This corresponds to the choice $b=0$ and \begin{equation}\label{eq:s8h98s78gd}(A_{i:})^\top = f_i - \frac{1}{d_i}\sum_{j \in N(i)}f_j.\end{equation} Note that $A\in \R^{n\times n}$.

\item {\em Spanning subgraph.} Let $(V,E')$ be any connected subgraph of $(V,E)$. For instance, we can choose a spanning tree. We can now apply any of the 2 models above to this new graph and either require $x_i=x_j$ for all $\{i,j\}\in E'$, or require the value $x_i$ to be equal to the average of the values $x_j$ for all neighbours $j$ of $i$ in $(V,E')$.
 
\end{enumerate}

Clearly, the above list does not exhaust the ways in which the constraint $x_1=\dots=x_n$ can be modeled as a linear system. For instance, we could build the system from constraints such as $x_1 = x_2 + x_4 -  x_3$, $x_1 = 5 x_2 - 4 x_7$ and so on.

Different representations of the constraint $x_1=\cdots=x_n$, in combination with a choice of $\cal D$, will lead to a wide range of specific algorithms for the consensus problem \eqref{eq:ohs09hud98yd}. Some (but not all) of these algorithms will have the property that communication only happens along the edges of the network, and these are the ones we are interested in. The number of combinations is very vast. We will therefore only highlight two options, with the understanding that based on this, the interested reader can assemble other specific methods as needed.

\subsection{Model 1: Each node is equal to its neighbours}

 Let $b=0$ and $A$ be as in \eqref{eq:89hs87s8ys}. Let the distribution $\cal D$ be defined by setting $S=e_i$  with probability $p_i>0$, where $e_i$ is the $i^{\text{th}}$ unit coordinate vector in $\R^m$. We have $B=I$, which means that Algorithm~\ref{alg:SDA-Primal} is the randomized Kaczmarz (RK) method \eqref{eq:SDA-primal009s09us0098} and Algorithm~\ref{alg:SDA} is the randomized coordinate ascent method  \eqref{eq:SDA-compact08986986098}. 

Let us take $y^0 = 0$ (which means that $x^0=c$), so that in Theorem~\ref{theo:Enormerror} we have $t=0$, and hence $x^k  \to x^*$. The particular choice of the starting point $x^0=c$ in the primal process has a very tangible meaning: for all $i$, node $i$ initially knows value $c_i$. The primal iterative process will dictate how the local values are modified in an iterative fashion so that eventually all nodes contain the optimal value $x^*_i = \bar{c}$.

\paragraph{Primal method.} In view of \eqref{eq:89hs87s8ys}, for each edge $e = (i,j)\in E$, we have  $\|A_{e:}\|^2=2$ and $A_{e:}x^k = x^k_i - x^k_j$. Hence, if the edge $e$ is selected by the RK method, \eqref{eq:SDA-primal009s09us0098} takes the specific form
\begin{equation}\label{eq:s98g98gsf66r6fs}\boxed{\quad x^{k+1} = x^k - \frac{x^k_i - x^k_j}{2} (f_i - f_j) \quad }\end{equation}
From \eqref{eq:s98g98gsf66r6fs} we see that only the $i^{\text{th}}$ and $j^{\text{th}}$ coordinates of $x^k$ are updated, via
\[x^{k+1}_i = x^k_i - \frac{x^k_i - x^k_j}{2} = \frac{x_i^k + x_j^k}{2}\]
and
\[x^{k+1}_j = x^k_j + \frac{x^k_i - x^k_j}{2} = \frac{x_i^k + x_j^k}{2}.\]
Note that in each iteration of RK, a random edge is selected, and the nodes on this edge replace their local values by their average. This is a basic variant of the {\em randomized gossip} algorithm \cite{RandGossip2006, ZouziasFreris2009}. 

\paragraph{Invariance.}
Let $f$ be the vector of all ones in $\R^n$ and notice that from \eqref{eq:s98g98gsf66r6fs} we obtain
$f^\top x^{k+1} = f^\top x^k$ for all $k$. This means that  for all $k\geq 0$ we have the invariance property:
\begin{equation}\label{eq:invariance}\sum_{i=1}^n x_i^k = \sum_{i=1}^n c_i.\end{equation}

\paragraph{Insights from the dual perspective.} We can now bring new insight into the randomized gossip algorithm by considering the  dual iterative process.  The dual method \eqref{eq:SDA-compact08986986098} maintains  weights $y^k$ associated with the edges of $E$ via the process:
\[y^{k+1} = y^k - \frac{A_{e:} (c - A^\top y^k)}{2}e_e,\]
where $e$ is a randomly selected edge. Hence, only the weight of a single edge is updated in each iteration. At optimality, we have
$x^* =c + A^\top y^*$. That is, for each $i$
\[\delta_i\eqdef \bar{c} - c_i = x_i^* - c_i =  (A^\top y^*)_i = \sum_{e\in E} A_{ei} y^*_e,\]
where $\delta_i$ is the correction term which needs to be added to $c_i$ in order for node $i$ to contain the value $\bar{c}$. From the above we observe that these correction terms are maintained by the dual method as an inner product of the $i^{\text{th}}$ column of $A$ and $y^k$, with the optimal correction being $\delta_i = A_{:i}^\top y^*$.

\paragraph{Rate.} Both Theorem~\ref{theo:Enormerror} and Theorem~\ref{theo:2} hold, and hence we automatically get several types of convergence for the randomized gossip method. In particular, to the best of our knowledge, no primal-dual type of convergence exist in the literature.  Equation~\eqref{eq:dualitygap09709709}  gives a stopping criterion certifying convergence via the duality gap, which is also new. 

In view of \eqref{eq:98hs8h8ss} and \eqref{eq:Laplacian09709}, and since $\|A\|_F^2 = 2m$, the convergence rate appearing in  all these complexity results is given by
\[\rho = 1 - \frac{\lambda_{\min}^+(L)}{2m},\]
where $L$ is the Laplacian of $(V,E)$. While it is know that the Laplacian is singular, the rate depends on the smallest nonzero eigenvalue. This means that the number of iterations needed to output an $\epsilon$-solution in expectation scales as $O(\left(2m/\lambda_{\min}^+(L)\right)\log(1/\epsilon))$, i.e., linearly with the number of edges.

\subsection{Model 2: Each node is equal to the average of its neighbours}

Let $A$ be as in \eqref{eq:s8h98s78gd} and $b=0$. Let the distribution $\cal D$ be defined by setting $S=f_i$  with probability $p_i>0$, where $f_i$ is the $i^{\text{th}}$ unit coordinate vector in $\R^n$. Again, we have $B=I$, which means that Algorithm~\ref{alg:SDA-Primal} is the randomized Kaczmarz (RK) method \eqref{eq:SDA-primal009s09us0098} and Algorithm~\ref{alg:SDA} is the randomized coordinate ascent method  \eqref{eq:SDA-compact08986986098}. As before, we choose $y^0=0$, whence $x^0=c$.

\paragraph{Primal method.} Observe that $\|A_{i:}\|^2 = 1 + 1/d_i$. The RK method \eqref{eq:SDA-primal009s09us0098} applied to this formulation of the problem takes the form
\begin{equation} \label{eq:sihiuhd098d7d}
\boxed{\quad x^{k+1} =  x^k - \frac{x^k_i - \frac{1}{d_i}\sum_{j\in N(i)}x^k_j  }{1 + 1/d_i} \left(f_i - \frac{1}{d_i}\sum_{j\in N(i)}f_j \right) \quad }
\end{equation}
where $i$ is chosen at random. This means that only coordinates in $i\cup N(i)$ get updated in such an iteration, the others remain unchanged. For node $i$ (coordinate $i$), this update is

\begin{equation}\label{eq:9g8g98gssdd} x^{k+1}_i = \frac{1}{d_i+1} \left( x_i^k + \sum_{j\in N(i)}x^k_j \right).\end{equation}
That is, the updated value at node $i$ is the average of the values of its neighbours and the previous value at $i$. From \eqref{eq:sihiuhd098d7d} we see that the values at nodes $j\in N(i)$ get updated as follows:

\begin{equation}\label{eq:98889ff} x^{k+1}_j = x_j^{k} + \frac{1}{d_i+1}\left(x_i^k - \frac{1}{d_i}\sum_{j'\in N(i)}x^k_{j'} \right).\end{equation}

\paragraph{Invariance.}  Let $f$ be the vector of all ones in $\R^n$ and notice that from \eqref{eq:sihiuhd098d7d} we obtain
\[f^\top x^{k+1} = f^\top x^k -\frac{x^k_i - \frac{1}{d_i}\sum_{j\in N(i)}x^k_j  }{1 + 1/d_i} \left(1 - \frac{d_i}{d_i} \right)=f^\top x^k, \]
for all $k$.

It follows that the method satisfies the invariance property \eqref{eq:invariance}.

\paragraph{Rate.} The method converges with the rate $\rho$ given by \eqref{eq:98hs8h8ss}, where $A$ is given by \eqref{eq:s8h98s78gd}.   If $(V,E)$ is a complete graph (i.e., $m=\tfrac{n(n-1)}{2}$), then $L = \tfrac{(n-1)^2}{n}A^\top A$ is the Laplacian. In that case,  $\|A\|_F^2 = \Tr{A^\top A} = \tfrac{n}{(n-1)^2}\Tr{L}=\tfrac{n}{(n-1)^2}\sum_i d_i  = \tfrac{n^2}{n-1}$ and hence
\[\rho \overset{\eqref{eq:s8h98s78gd}}{=} 1 - \frac{\lambda_{\min}^+(A^\top A)}{\|A\|_F^2} = 1 - \frac{\tfrac{n}{(n-1)^2}\lambda_{\min}^+(L)}{\tfrac{n^2}{n-1}} = 1 - \frac{\lambda_{\min}^+(L)}{2m}.\]

\section{Proof of Theorem~\ref{theo:Enormerror}} \label{sec:proof}

In this section we prove Theorem~\ref{theo:Enormerror}. We proceed as follows:  in Section~\ref{subsec:error} we characterize the space in which the iterates move, in Section~\ref{subsec:inequality} we establish a certain key technical  inequality, in Section~\ref{subsec:convergence} we establish  convergence of iterates, in Section~\ref{subsec:residual} we derive a rate for the residual and finally, in Section~\ref{subsec:lower_bound} we establish the lower bound on the convergence rate.

\subsection{An error lemma} \label{subsec:error}

The following result describes the space in which the iterates move.   It is an extension of the observation in Remark~\ref{lem:5shsuss} to the case when $x^0$ is chosen arbitrarily.

\begin{lemma} \label{lem:error} Let the assumptions of Theorem~\ref{theo:Enormerror} hold. For all $k\geq 0$ there exists $w^k\in \R^m$ such that
$e^k \eqdef  x^k - x^* - t = B^{-1}A^\top w^k$.
\end{lemma}
\begin{proof}
We proceed by induction. Since by definition, $t$ is the projection of $x^0-c$ onto  $\Null{A}$ (see \eqref{eq:98hs8htt}), applying Lemma~\ref{eq:decomposition} we know that $x^0 -c = s+ t$, where $s = B^{-1}A^\top \hat{y}^0$ for some $\hat{y}^0\in \R^m$. Moreover, in view of \eqref{eq:opt_primal}, we know that $x^* =c + B^{-1}A^\top y^*$, where $y^*$ is any dual optimal solution. Hence, 
\[e^0  = x^0 - x^* - t = B^{-1}A^\top (\hat{y}^0-y^*).\]
Assuming the relationship holds for $k$, we have
\begin{eqnarray*}e^{k+1} & = & x^{k+1} - x^* - t\\
 &\overset{(\text{Alg}~\ref{alg:SDA-Primal})}{=} & \left[x^k - B^{-1}A^\top S (S^\top A B^{-1} A^\top S)^\dagger S^\top (A x^k  -b ) \right] - x^*-t\\
&= &  \left[x^* + t + B^{-1}A^\top w^k - B^{-1}A^\top S (S^\top A B^{-1} A^\top S)^\dagger S^\top (A x^k  -b )\right] - x^* - t\\
&=& B^{-1}A^\top w^{k+1},
\end{eqnarray*}
where $w^{k+1} = w^k - S (S^\top A B^{-1} A^\top S)^\dagger S^\top (A x^k  -b )$. \qed
\end{proof}

\subsection{A key inequality} \label{subsec:inequality}

The following inequality is of key importance in the proof of the main theorem.

\begin{lemma}\label{lem:WGWtight}
Let $0\neq W \in  \R^{m\times n}$ and
$G \in \R^{m\times m}$ be symmetric positive definite. Then the matrix $W^\top G W$ has a positive eigenvalue, and the following inequality holds for all $y\in \R^m$:
\begin{equation}\label{eq:WGWtight}
 y^\top WW^\top G WW^\top y \geq \lambda_{\min}^+(W^\top G W)\|W^\top y\|^2.
 \end{equation}
 Furthermore, this bound is tight.
\end{lemma}
\begin{proof} Fix arbitrary $y\in \R^m$. By Lemma~\ref{lem:09709s}, $W^\top y\in \Range{W^\top G W}$. Since, $W$ is nonzero, the positive semidefinite matrix $W^\top G W$ is also nonzero, and hence it has a positive eigenvalue. Hence, $\lambda_{\min}^+(W^\top G W)$ is well defined. 
Let $\lambda_{\min}^+(W^\top G W) = \lambda_{1}\leq \cdots\leq  \lambda_{\tau}$ be the positive eigenvalues of $W^\top GW$, with associated orthonormal eigenvectors $q_1,\dots,q_\tau$. We thus have
\[W^\top G W = \sum_{i=1}^\tau \lambda_i q_i q_i^\top.\] It is easy to see that these eigenvectors span $\Range{W^\top GW}$. Hence,  we can write $W^\top y = \sum_{i=1}^{\tau} \alpha_i q_i$ and therefore
\[y^\top WW^\top G W W^\top y = \sum_{i=1}^{\tau} \lambda_{i}\alpha_i^2 \geq  
\lambda_1\sum_{i=1}^{\tau} \alpha_i^2 = \lambda_1 \|W^\top y\|^2.\]
Furthermore this bound is tight, as can be seen by selecting $y$ so that $W^\top y = q_1$. \qed
\end{proof}

\subsection{Convergence of the iterates} \label{subsec:convergence}

Subtracting $x^*+t$ from both sides of the update step of Algorithm~\ref{alg:SDA-Primal}, and letting \[Z= Z_{S^\top A} \overset{\eqref{eq:Z_A}}{=} A^\top S (S^\top A B^{-1} A^\top S)^\dagger S^\top  A,\] we obtain the identity
\begin{equation} \label{eq:fixed0}
 x^{k+1} - (x^{*}+t) = (I-B^{-1}Z)(x^k-(x^{*}+t)),
 \end{equation}
 where we used that $t\in \Null{A}$.  Let 
 \begin{equation}\label{eq:h98sh9sd}e^k = x^k - (x^{*}+t)\end{equation} and note that in view of \eqref{eq:H}, $\E{Z} = A^\top H A$. Taking norms and expectations (in $S$) on both sides of~\eqref{eq:fixed0}  gives 
\begin{eqnarray}
\E{\norm{e^{k+1}}_{B}^2\, | \, e^k} &=& 
\E{\norm{(I-{B^{-1}}Z)e^k }^2_B} \nonumber \\
&\overset{(\text{Lemma}~\ref{eq:decomposition}, \text{ Equation}~\eqref{eq:iuhiuhpp} )}{=}& \E{(e^k)^\top (B-Z) e^k} \nonumber\\
& =& \norm{e^k}_{B}^2-
(e^k)^\top \E{Z}e^k \nonumber \\
&=& \norm{e^k}_{B}^2- (e^k)^\top A^\top H A e^k, \label{eq:theostep1}
\end{eqnarray}
where in the second step we have used \eqref{eq:iuhiuhpp} from Lemma~\ref{eq:decomposition}  with $A$ replaced by $S^\top A$. In view of Lemma~\ref{lem:error}, let $w^k\in \R^{m}$ be such that $e^k = B^{-1}A^\top w^k.$  Thus
\begin{eqnarray}
(e^k)^\top A^\top H A e^k \quad &=&(w^k)^\top AB^{-1}A^\top H A B^{-1}A^\top w^k\nonumber \\
&\overset{(\text{Lemma}~\ref{lem:WGWtight})}{\geq}& \lambda_{\min}^+(B^{-1/2}A^\top H A B^{-1/2})\cdot
\|B^{-1/2}A^\top w^k\|^2 \nonumber\\
&= &(1-\rho) \cdot \|B^{-1}A^\top w^k\|_B^2 \\
&=& (1- \rho) \cdot\norm{e^k}_B^2,\label{eq:theostep2}
\end{eqnarray}
where we applied Lemma~\ref{lem:WGWtight} with $W = AB^{-1/2}$ and $G = H$,  so that $W^\top GW = B^{-1/2}A^\top H A B^{-1/2}.$
Substituting~\eqref{eq:theostep2} into~\eqref{eq:theostep1} gives
$\E{\norm{e^{k+1}}_{B}^2\, | \, e^k} \leq  \rho \cdot \norm{e^k}_B^2$.
Using the tower property of expectations, we obtain the recurrence
\[\E{\norm{e^{k+1}}_{B}^2} \leq  \rho \cdot \E{\norm{e^k}_B^2}. \]
To prove~\eqref{eq:Enormerror} it remains to  unroll the recurrence.

\subsection{Convergence of the residual} \label{subsec:residual}

We now prove \eqref{eq:s98h09hsxxx}. Letting $V_k = \|x^k-x^*-t\|_B^2$, we have
\begin{eqnarray*}
\E{\|Ax^k - b\|_B} &=& \E{\|A(x^k-x^*-t) + At\|_B}\\
&\leq &\E{\|A(x^k-x^*-t) \|_B} + \|At\|_B\\
&\leq &  \|A\|_B \E{\sqrt{V_k}} + \|At\|_B\\
&\leq &\|A\|_B \sqrt{\E{V_k}} + \|At\|_B\\
&\overset{\eqref{eq:Enormerror}}{\leq}&\|A\|_B \sqrt{\rho^k V_0 } + \|At\|_B,
 \end{eqnarray*} 
 where in the step preceding the last one we have used Jensen's inequality.

\subsection{Proof of the lower bound}\label{subsec:lower_bound}

Since $\E{Z} = A^\top H A$, using Lemma~\ref{lem:09709s} with $G=H$ and $W=A B^{-1/2}$ gives
\begin{align*}
\Range{B^{-1/2}\E{Z}B^{-1/2}}=\Range{B^{-1/2}A^\top},
\end{align*}
from which we deduce that \begin{eqnarray*}\Rank{A} &=&
 \dim\left(\Range{A^\top}\right) \\
&=& \dim\left(\Range{B^{-1/2}A^\top}\right)\\
&=&  \dim\left(\Range{B^{-1/2}\E{Z}B^{-1/2}}\right)\\
&=& \Rank{B^{-1/2}\E{Z}B^{-1/2}}.\end{eqnarray*}

Hence, $\Rank{A}$ is equal to the number of nonzero eigenvalues of $B^{-1/2}\E{Z}B^{-1/2}$, from which we immediately obtain the bound
\begin{eqnarray*}
\Tr{B^{-1/2}\E{Z}B^{-1/2}} &\geq& \Rank{A }\, \lambda_{\min}^+(B^{-1/2}\E{Z}B^{-1/2}). 
\end{eqnarray*}
In order to obtain \eqref{eq:nubound},  it only remains to combine the above inequality with
\[
\E{\Rank{S^\top A}}  \overset{ \eqref{eq:ugisug7sss}}{=} \E{\Tr{B^{-1}Z}}  =  \E{\Tr{B^{-1/2}Z B^{-1/2}}} \\
= \Tr{B^{-1/2}\E{Z}B^{-1/2}}. 
\]

\section{Proof of Theorem~\ref{theo:2}} \label{sec:proof2}

In this section we prove Theorem~\ref{theo:2}. We dedicate a subsection to each of the three complexity bounds.

\subsection{Dual suboptimality}

Since $x^0\in c+\Range{B^{-1}A^\top}$, we have $t=0$ in Theorem~\ref{theo:Enormerror}, and hence \eqref{eq:Enormerror} says that \begin{equation}\label{eq:s98h98shs}\E{U_k} \leq \rho^k U_0.\end{equation}
It remains to apply Proposition~\ref{lem:correspondence}, which says that $U_k = D(y^*)-D(y^k)$.

\subsection{Primal suboptimality}

Letting $U_k = \tfrac{1}{2}\|x^k- x^*\|_B^2$, we can write
\begin{eqnarray}
P(x^k) - OPT &=& \tfrac{1}{2}\|x^k - c\|_B^2 - \tfrac{1}{2}\|x^* - c\|_B^2\notag\\
&=& \tfrac{1}{2}\|x^k - x^* + x^* - c\|_B^2  - \tfrac{1}{2}\|x^* - c\|_B^2\notag\\
&=& \tfrac{1}{2}\|x^k- x^*\|_B^2 +  (x^k-x^*)^\top B (x^*-c)\notag\\
&\leq & U_k + \|x^k-x^*\|_B \|B(x^*-c)\|_{B^{-1}} \notag\\
&=& U_k +\|x^k-x^*\|_B  \|x^*-c\|_B \notag \\
&= & U_k + 2 \sqrt{U_k} \sqrt{OPT}.\label{eq:iuhs89h98s6s}
\end{eqnarray}
By taking expectations on both sides of \eqref{eq:iuhs89h98s6s}, and using Jensen's inequality, we therefore obtain
\[\E{P(x^k)-OPT} \leq \E{U_k} + 2\sqrt{OPT} \sqrt{\E{U_k}} \overset{\eqref{eq:s98h98shs}}{\leq} \rho^k U_0 + 2 \rho^{k/2} \sqrt{OPT \times U_0},\]
which establishes the bound on primal suboptimality \eqref{eq:PRIMALSUBOPT}. 

\subsection{Duality gap}

Having established rates for primal and dual suboptimality, the rate for the duality gap follows easily:
\begin{eqnarray*}
\E{P(x^k) - D(y^k)} &=& \E{P(x^k) - OPT + OPT - D(y^k)}\\
&=&  \E{P(x^k) - OPT} + \E{OPT - D(y^k)}\\
&\overset{\eqref{eq:DUALSUBOPT}  + \eqref{eq:PRIMALSUBOPT}}{=}& 2 \rho^k U_0 + 2 \rho^{k/2} \sqrt{OPT \times U_0}.
\end{eqnarray*}

\section{Numerical Experiments: Randomized Kaczmarz Method with Rank-Deficient System}

To illustrate some of the novel aspects of our theory, we perform numerical experiments with the Randomized Kaczmarz method~\eqref{eq:SDA-primal009s09us0098} (or equivalently the randomized coordinate ascent method applied to the dual problem~\eqref{eq:D}) and compare the empirical convergence to the convergence predicted by our theory. We test several randomly generated rank-deficient systems and compare the evolution of the empirical primal error $\norm{x^k-x^*}_2^2/\norm{x^0-x^*}_2^2$ to the convergence dictated by the rate $\rho = 1-\lambda_{\min}^+\left(A^\top A\right)/\norm{A}_F^2$ given in~\eqref{eq:98hs8h8ss} and the lower bound $1-1/\Rank{A}\leq\rho$. From Figure~\ref{fig:rand} we can see that the RK method converges despite the fact that the linear systems are rank deficient. While previous results do not guarantee that RK converges for rank-deficient matrices, our theory does as long as the system matrix has no zero rows. Furthermore, we observe that the lower the rank of the system matrix, the faster the  convergence of the RK method, and moreover, the closer the empirical convergence is to the convergence dictated by the rate $\rho $ and lower bound on $\rho$. In particular, on the low rank system in Figure~\ref{fig:randa}, the empirical convergence is very close to both the convergence dictated by $\rho$ and the lower bound.
While on the full rank system in Figure~\ref{fig:randd}, the convergence dictated by $\rho$ and the lower bound on $\rho$ are no longer an accurate estimate of the empirical convergence.

\begin{figure}[!h]
    \centering
\begin{subfigure}[t]{0.45\textwidth}
        \centering
\includegraphics[width =  \textwidth, trim= 80 270 80 280, clip ]{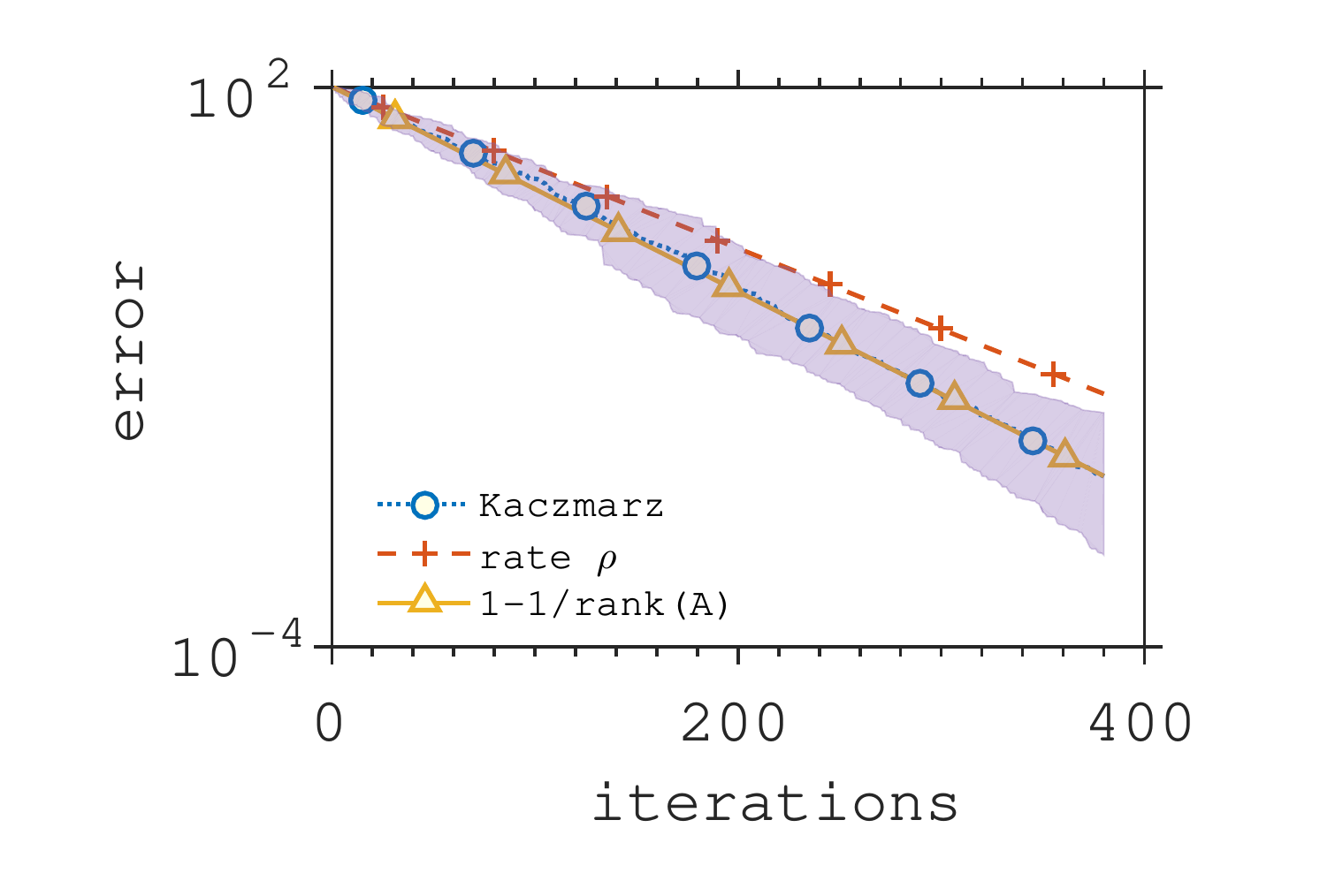}
        \caption{$\Rank{A}= 40$}\label{fig:randa}
\end{subfigure}%
\begin{subfigure}[t]{0.45\textwidth}
        \centering
\includegraphics[width =  \textwidth, trim= 80 270 80 280, clip ]{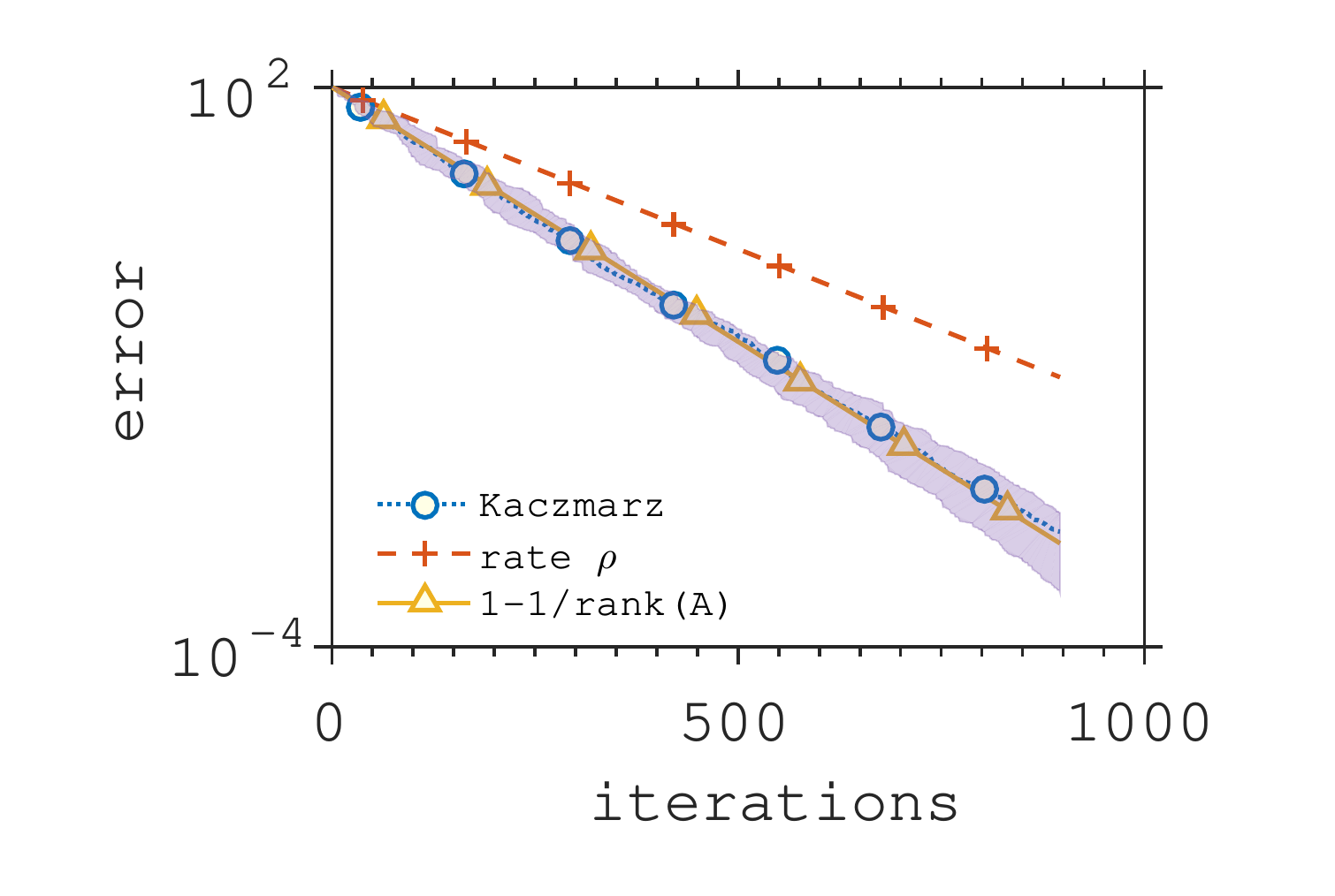}
        \caption{$\Rank{A} = 80$}\label{fig:randb}
\end{subfigure}\\%
\begin{subfigure}[t]{0.45\textwidth}
        \centering
\includegraphics[width =  \textwidth, trim= 80 270 80 280, clip ]{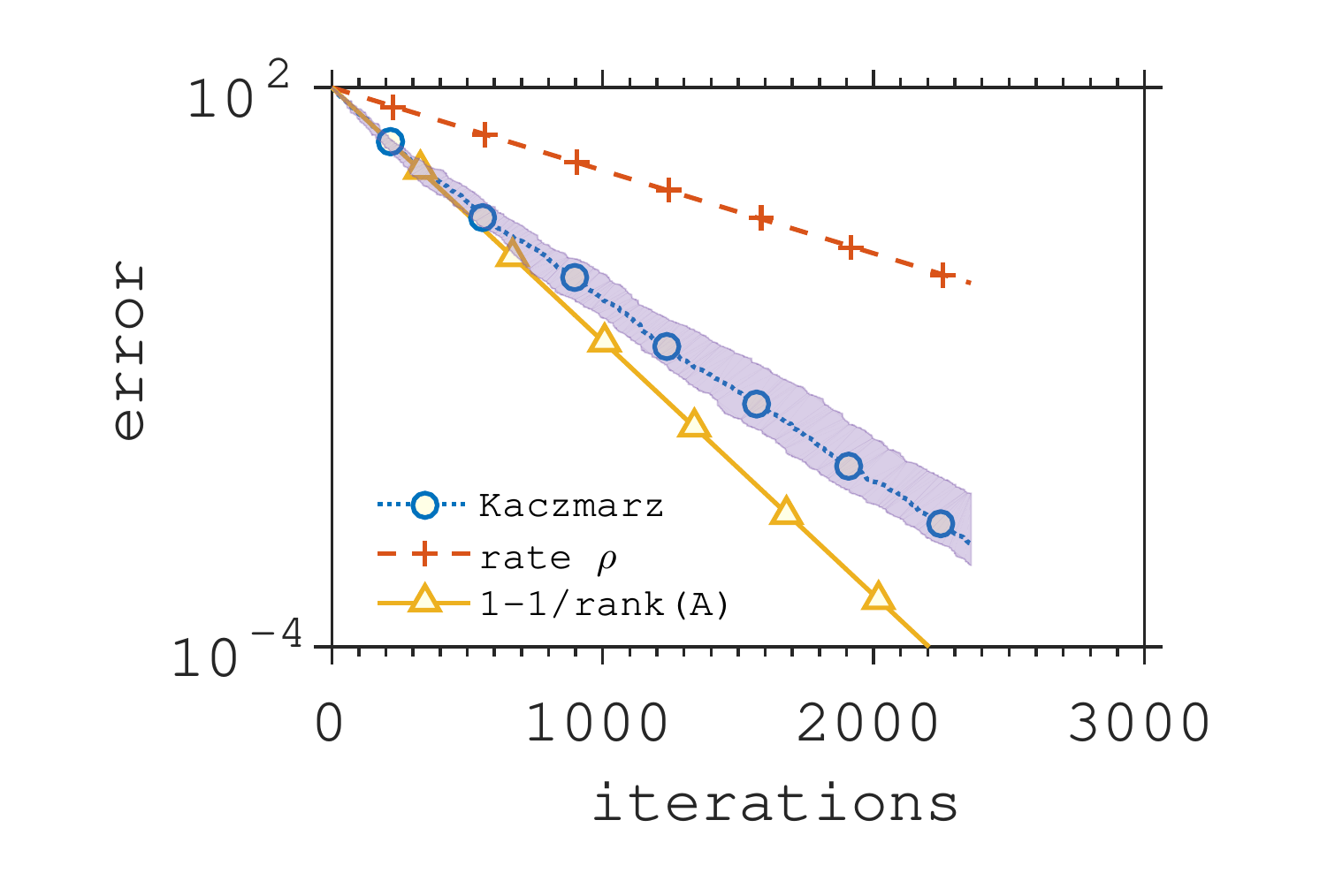}
        \caption{$\Rank{A} = 160$}\label{fig:randc}
\end{subfigure}%
\begin{subfigure}[t]{0.45\textwidth}
        \centering
\includegraphics[width =  \textwidth, trim= 80 270 80 280, clip ]{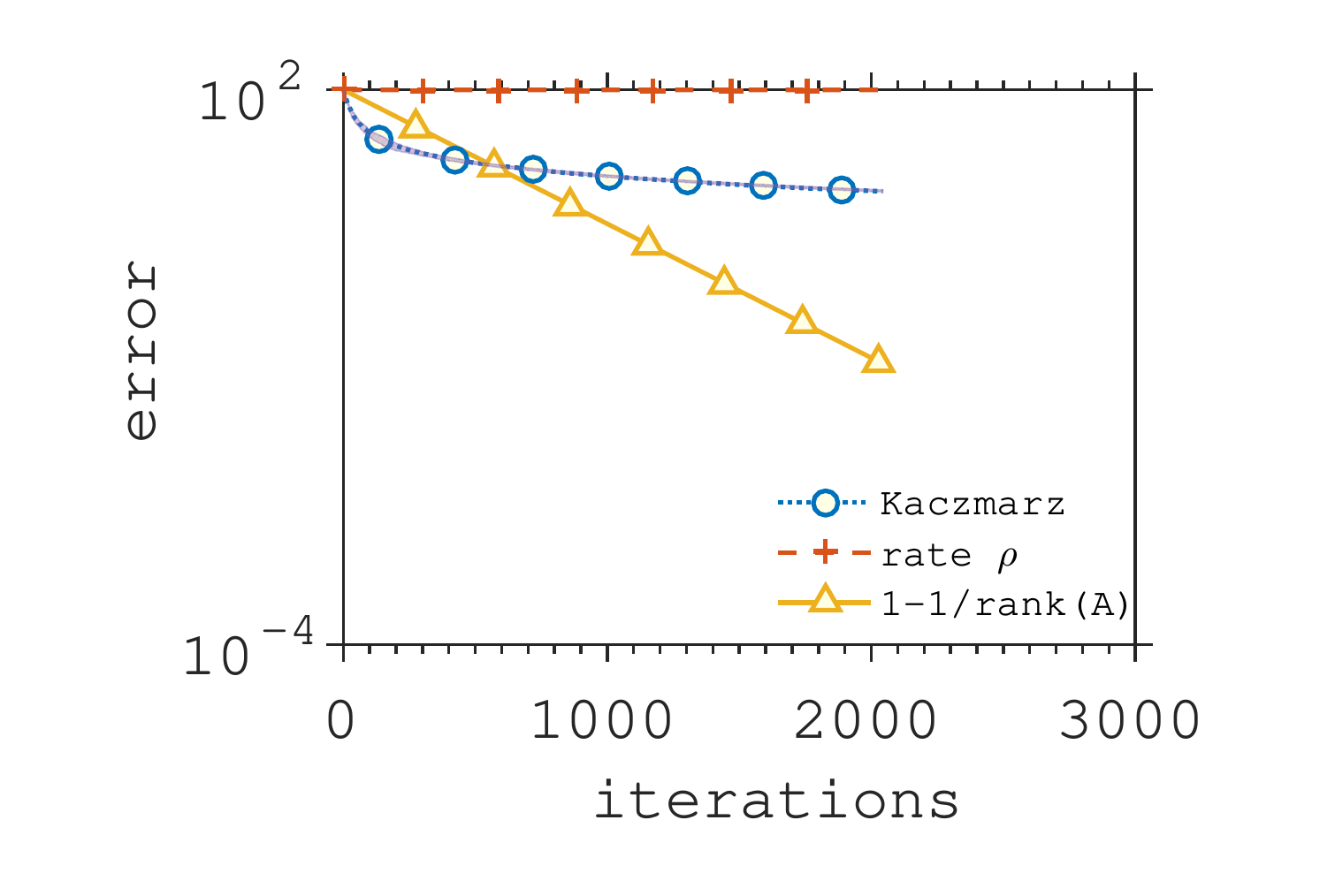}
        \caption{$\Rank{A} = 300$}\label{fig:randd}
\end{subfigure}%
    \caption{Synthetic MATLAB generated problems.  Rank deficient matrix $A~=~\sum_{i=1}^{\Rank{A}} \sigma_i u_i v_i^T $ where $\sum_{i=1}^{300} \sigma_i u_i v_i^T =$\texttt{rand}$(300,300)$ is an svd decomposition of a $300\times 300$ uniform random matrix. We repeat each experiment ten times. The blue shaded region is the $90\%$ percentile of relative error achieved in each iteration.
    }\label{fig:rand}
\end{figure}

\section{Conclusion} \label{sec:conclusion}

We have developed a versatile and powerful  algorithmic framework for solving linear systems: {\em stochastic dual ascent (SDA)}. In particular, SDA finds the projection of a given point, in a fixed but arbitrary Euclidean norm, onto the solution space of the system.  Our method is  dual in nature, but can also be described in terms of primal iterates via a simple affine transformation of the dual variables. Viewed as a dual method, SDA belongs to a  novel class of randomized optimization algorithms:  it updates the current iterate by adding the product of a random matrix, drawn independently from a fixed distribution, and a vector. The update is chosen as the best point lying in the random subspace spanned by the columns of this random matrix. 

While SDA is the first method of this type, particular choices for the distribution of the random matrix lead to several known algorithms: randomized coordinate descent \cite{Leventhal:2008:RMLC} and randomized Kaczmarz \cite{SV:Kaczmarz2009} correspond to a discrete distribution over the columns of the identity matrix, randomized Newton method \cite{SDNA} corresponds to a discrete distribution over column submatrices of the identity matrix, and Gaussian descent \cite{S.U.StichC.L.Muller2014} corresponds to the case when the random matrix is a  Gaussian vector. 

We equip the method with several complexity results with the same rate of exponential decay in expectation (aka linear convergence) and establish a tight lower bound on the rate. In particular, we prove convergence of primal iterates, dual function values, primal function values, duality gap and of the residual. The method converges under very weak conditions beyond consistency of the linear system. In particular, no rank assumptions on the system matrix are needed. For instance,   randomized Kaczmarz method converges linearly as long as the system matrix contains no zero rows.

Further, we show that SDA can be applied to the distributed (average) consensus problem. We recover a standard randomized gossip algorithm as a special case, and show that its complexity is proportional to the number of edges in the graph and inversely proportional to the smallest nonzero eigenvalue of the graph Laplacian.  Moreover, we illustrate how our framework can be used to obtain new randomized algorithms for the distributed consensus problem. 

Our framework extends to several other problems in optimization and numerical linear algebra. For instance, one can apply it to develop new stochastic algorithms for computing the inverse of a matrix and obtain state-of-the art performance for inverting matrices of huge sizes.

\section{Acknowledgements}

The second author would like to thank Julien Hendrickx from Universit\'{e} catholique de Louvain for a discussion regarding randomized gossip algorithms.

{\small
 \printbibliography
 }

\section{Appendix: Elementary Results Often Used in the Paper}

We first state a lemma comparing the null spaces and range spaces of certain related matrices.  While the result is of an elementary nature, we use it several times in this paper, which justifies its elevation to the status of a lemma. The proof is brief and hence we include it for completeness.

\begin{lemma}\label{lem:09709s}For any matrix $W$ and symmetric positive definite matrix $G$,
\begin{equation}\label{eq:8ys98hs}\Null{W} = \Null{W^\top G W}\end{equation}
and
\begin{equation}\label{eq:8ys98h986ss}\Range{W^\top } = \Range{W^\top G W}.\end{equation}
\end{lemma}
\begin{proof} 
In order to establish \eqref{eq:8ys98hs}, it suffices to show the inclusion $\Null{W} \supseteq \Null{W^\top G W}$ since the reverse inclusion trivially holds. Letting $s\in \Null{W^\top G W}$, we see that $\|G^{1/2}Ws\|^2=0$, which implies $G^{1/2}Ws=0$. Therefore, $s\in \Null{W}$. Finally, \eqref{eq:8ys98h986ss} follows from \eqref{eq:8ys98hs} by taking orthogonal complements. Indeed, $\Range{W^\top}$ is the orthogonal complement of $\Null{W}$ and $\Range{W^\top G W}$ is the orthogonal complement of $\Null{W^\top G W}$. \qed
\end{proof}

\bigskip
The following technical lemma is a variant of a standard  result of linear algebra (which is recovered in the $B=I$ case). While the results are folklore and easy to establish, in the proof of our main theorem we need certain details which are hard to find in textbooks on linear algebra, and hence hard to refer to. For the benefit of the reader, we include the  detailed statement and proof.
 
\begin{lemma} [Decomposition and Projection]\label{eq:decomposition} Each $x\in \R^n$ can be decomposed in a unique way as $x = s(x) + t(x)$, where $s(x)\in \Range{B^{-1}A^\top}$  and $t(x)\in \Null{A}$. Moreover, the decomposition can be computed explicitly as
\begin{equation}
\label{eq:98hs8hss}s(x) =  \arg \min_{s} \left\{ \|x-s\|_B \;:\; s\in \Range{B^{-1}A^\top} \right\}=  B^{-1} Z_A x \end{equation}
and
\begin{equation}
\label{eq:98hs8htt}t(x) = \arg \min_{t} \left\{ \|x-t\|_B \;:\; t\in \Null{A} \right\}= (I - B^{-1}Z_A) x,\end{equation}
where
\begin{equation}\label{eq:Z_A}Z_A\eqdef A^\top (AB^{-1}A^\top)^\dagger  A.\end{equation}
Hence, the matrix $B^{-1}Z_A$ is a projector in the $B$-norm onto $\Range{B^{-1}A^\top}$, and $I-B^{-1}Z_A$ is a projector in the $B$-norm onto $\Null{A}$. Moreover, for all $x\in \R^n$ we have $\|x\|_B^2 = \|s(x)\|_B^2 + \|t(x)\|_B^2$, with
\begin{equation}\label{eq:iuhiuhpp}\|t(x)\|_B^2 = \|(I-B^{-1}Z_{A})x\|_B^2 = x^\top (B-Z_{ A}) x\end{equation}
and
\begin{equation}\label{eq:iuhiuhppss}\|s(x)\|_B^2 = \|B^{-1}Z_{A} x\|_B^2 = x^\top Z_{ A} x.\end{equation}
Finally, 
\begin{equation}\label{eq:ugisug7sss}\Rank{A} = \Tr{B^{-1}Z_A}.\end{equation}
\end{lemma}

\begin{proof} Fix arbitrary $x\in \R^n$. We first establish existence of the decomposition. By Lemma~\ref{lem:09709s} applied to $W=A^\top$ and $G=B^{-1}$ we know that there exists $u$ such that $Ax = A B^{-1}A^\top u$. Now let $s = B^{-1}A^\top u$ and $t = x-s$. Clearly, $s\in \Range{B^{-1}A^\top}$ and $t\in \Null{A}$.  For  uniqueness, consider two decompositions: $x = s_1+ t_1$ and $x=s_2 + t_2$. Let $u_1,u_2$ be vectors such that $s_i = B^{-1}A^\top u_i$, $i=1,2$. Then $AB^{-1}A^\top(u_1-u_2)=0$. Invoking Lemma~\ref{lem:09709s} again, we see that $u_1-u_2\in \Null{A^\top}$, whence $s_1 = B^{-1}A^\top u_1 = B^{-1}A^\top u_2 = s_2$. Therefore, $t_1 = x - s_1 = x-s_2 = t_2$, establishing uniqueness.

Note that $s = B^{-1}A^\top y$, where $y$ is any solution of the optimization problem 
\[\min_y \tfrac{1}{2}\|x-B^{-1}A^\top y\|_B^2.\]
The first order necessary and sufficient optimality conditions are $Ax = AB^{-1}A^\top y$. In particular, we may choose $y$ to be the least norm solution of this system, which gives $y=(AB^{-1}A^\top)^\dagger Ax$, from which \eqref{eq:98hs8hss} follows. The variational formulation \eqref{eq:98hs8htt}  can be established in  a similar way, again via first order optimality conditions (note that the closed form formula \eqref{eq:98hs8htt} also directly follows from \eqref{eq:98hs8hss} and the fact that $t = x - s$). 

Next,  since $x=s+t$ and $s^\top B t = 0$,  
\begin{equation}\label{eq:09u0hss}
\|t\|_B^2=(t+s)^\top B t = x^\top B t \overset{\eqref{eq:98hs8htt}}{=} x^\top B (I-B^{-1}Z_{ A})x
= x^\top (B - Z_{ A}) x
\end{equation}
and
\[ \|s\|_B^2 = \|x\|_B^2 - \|t\|_B^2  \overset{\eqref{eq:09u0hss}}{=} x^\top Z_A x.\]

It only remains to establish \eqref{eq:ugisug7sss}. Since  $B^{-1}Z_A$ is a projector onto $\Range{B^{-1}A^\top}$ and since the trace of each projector is equal to the dimension of the space they project onto, we have $\Tr{B^{-1}Z_A} = \dim(\Range{B^{-1}A^\top}) = \dim(\Range{A^\top}) = \Rank{A}$.\qed

\end{proof}

\end{document}